%% file: cas-dc-bundle_07_05_2026_clean.tex
\documentclass[a4paper,fleqn]{cas-dc}
\input{def}

\usepackage[numbers]{natbib}
\usepackage{float}
\usepackage{tikz}
\usepackage{caption}
\usepackage{amsmath}
\usepackage{pgfplots}
\usepackage{hyperref}
\pgfplotsset{compat=1.18}

\def\tsc#1{\csdef{#1}{\textsc{\lowercase{#1}}\xspace}}
\tsc{WGM}
\tsc{QE}

\begin{document}
\let\WriteBookmarks\relax
\def\floatpagepagefraction{1}
\def\textpagefraction{.001}

\shorttitle{}    


\title [mode = title]{Complexity and numerical experiments of a new adaptive generic proximal bundle method}  



%

\author[1]{Vincent Guigues}
\ead{vincent.guigues@fgv.br}






\affiliation[1]{
organization={School of Applied Mathematics, FGV},
addressline={Praia de Botafogo},
city={Rio de Janeiro},
country={Brazil}
}

\author[2]{Renato D.C. Monteiro}





\affiliation[2]{
organization={Georgia Institute of Technology},
addressline={Atlanta, Georgia 30332-0205},
country={USA}
}

\author[3]{Benoit Tran}






\affiliation[3]{organization={School of Applied Mathematics, FGV},
addressline={Praia de Botafogo},
city={Rio de Janeiro},
country={Brazil}
}

\author[4]{Adriana W. Henarejos}






\affiliation[4]{organization={School of Applied Mathematics, FGV},
addressline={Praia de Botafogo},
city={Rio de Janeiro},
country={Brazil}
}

\cortext[1]{Corresponding author.}



\begin{abstract}
This paper 
develops an adaptive generic proximal bundle method, shows its complexity, and presents numerical experiments comparing this method with other competing methods on a set of optimization problems. \\
\end{abstract}





\maketitle


\section{Introduction}

Existing proximal bundle methods for hybrid composite optimization problems typically rely on fixed proximal parameters or require generic solvers, which may limit their practical efficiency. This paper studies adaptive proximal bundle methods for solving convex hybrid composite optimization (HCO) problems. The novelty of this work lies in the introduction of adaptive mechanisms within Generic Proximal Bundle (GPB) variants based on low-memory one-cut and two-cuts models, together with a solver-free implementation. The proposed approach preserves the structure of these classical models while dynamically adjusting the proximal parameter during the iterations, leading to efficient implementations. Beyond theoretical guarantees, we demonstrate  the practical efficiency of the Adaptative GPB through extensive numerical experiments. We now introduce the notation, assumptions, and basic definitions that will be used throughout the paper.

	Let $\R^n$ denote the standard $n$-dimensional Euclidean space equipped with  inner product and norm denoted by $\left\langle \cdot,\cdot\right\rangle $
	and $\|\cdot\|$, respectively. 
	Let $\Psi: \R^n\rightarrow (-\infty,+\infty]$ be given. Let $\dom \Psi:=\{x \in \R^n: \Psi (x) <\infty\}$ denote the effective domain of $\Psi$.
 We say that
	$\Psi$ is proper if $\dom \Psi \ne \emptyset$.
	A proper function $\Psi: \R^n\rightarrow (-\infty,+\infty]$ is $\mu$-convex for some $\mu \ge 0$ if
	$$
	\Psi(\alpha z+(1-\alpha) z')\leq \alpha \Psi(z)+(1-\alpha)\Psi(z') - \frac{\alpha(1-\alpha) \mu}{2}\|z-z'\|^2
	$$
	for every $z, z' \in \dom \Psi$ and $\alpha \in [0,1]$.
 The set of all proper lower semicontinuous $\mu$-convex functions is denoted by $\mConv{n}$.
 When $\mu=0$, we simply denote
    $\mConv{n}$ by $\bConv{n}$.
	The subdifferential of $\Psi$ at $z \in \dom \Psi$ is denoted by $\partial \Psi (z)$.

This paper considers convex hybrid composite optimization (HCO) problems of form
\begin{equation}\label{eq:ProbIntro2}
	\phi_{*}:=\min \left\{\phi(x):=f(x)+h(x): x \in \R^n\right\},
	\end{equation}
where the following conditions hold for some triple
$(\overline L_f,\allowbreak \overline M_f,\allowbreak \mu)\in \mathbb{R}_+^3$ for functions $f$ and $h$:
\begin{itemize}
    \item[(A1)]
    $f \in \bConv{n}$ and $h \in \mConv{n}$ with  are such that
    $\dom h \subset \dom f$, and a subgradient oracle,
     i.e.,
    a function $f':\dom h \to \R^n$
    satisfying $f'(x) \in \partial f(x)$ for every $x \in \dom h$, is available;
    \item[(A2)]
    the set of optimal solutions $X^*$ of
    problem \eqref{eq:ProbIntro2} is \allowbreak nonempty;
    \item[(A3)]
    for every $u,v \in \dom h$,
    \[
    \|f'(u)-f'(v)\| \le 2{\overline M}_f + {\overline L}_f \|u-v\|.
    \]
\end{itemize}

 This work is concerned with universal methods, i.e., \allowbreak parameter-free methods for solving \eqref{eq:ProbIntro2} that do not require knowledge of any of parameters associated with the instance $(f,h)$, such as
  the parameter pair $(\overline M_f,\overline L_f)$ or the intrinsic convexity parameter $\mu$ of $h$.
The first universal methods for solving \eqref{eq:ProbIntro2} under the condition that $\nabla f$ is Hölder continuous 
have been presented in
\cite{nesterov2015universal}
and 
\cite{lan2015bundle}.
Additional universal methods
for solving \eqref{eq:ProbIntro2} have been studied in \cite{liang2021average,liang2023average,mishchenko2020adaptive,zhou2024adabb} under the condition that $f$ is smooth,
and  in \cite{li2023simple,liang2020proximal,liang2023unified} for the case where $f$ is either smooth or nonsmooth.

In this paper, we introduce Adaptive GPB, an extension of the Generic Proximal Bundle (GPB) method from \cite{liang2023unified}. While GPB is an inexact proximal point method with a fixed stepsize $\lambda$, Adaptive GPB dynamically updates $\lambda$ along the iterations based on the observed behavior of the algorithm.

The proposed adaptive mechanism is designed to reduce sensitivity to the initial stepsize and to handle different convexity regimes more robustly. This is achieved by organizing the algorithm into cycles, each associated with a proximal subroutine whose outcome is evaluated through an explicit contraction test. According to this test, a cycle is declared successful or unsuccessful, and the stepsize is increased or decreased accordingly, leading to a clear separation between local progress and global control.

The present work builds upon and extends the framework developed in \cite{liang2023unified}. While in \cite{liang2023unified} the contraction properties required for complexity analysis are assumed and guaranteed theoretically for a fixed stepsize, Adaptive GPB explicitly measures and exploits contraction through the ratio $\alpha_i$ computed at the end of each cycle. This explicit contraction criterion plays a central role in the adaptive mechanism and in the resulting complexity analysis. Moreover, while \cite{liang2023unified} considers OneCut, TwoCuts, and MultiCuts bundle update strategies, the present work focuses on the simpler and computationally efficient OneCut and TwoCuts variants, which are particularly well suited for an adaptive scheme.

A closely related contribution is \cite{guigues2024universal}, which introduces universal methods for hybrid composite optimization, including the universal proximal bundle method (U-PB). While both approaches leverage the functional framework for strongly convex optimization (FSCO), the adaptive mechanism of Adaptive GPB differs substantially. In U-PB, stepsize reduction acts as a safeguard after unsuccessful null steps and relies on a fixed bound on internal iterations per cycle. In contrast, Adaptive GPB regulates the stepsize directly at the cycle level via an explicit contraction test, allowing both increases and decreases without imposing a fixed iteration ceiling. This cycle-based structure requires additional arguments beyond those in \cite{guigues2024universal}.

We establish the complexity of Adaptive GPB and report numerical experiments comparing it with two classical bundle methods from \cite{LemarechalNewVariantsBundle1995,kiwielProximityControlBundle1990}, GPB (one-cut, two-cut and multicut versions) \cite{liang2023unified} and U-PB \cite{guigues2024universal}. The experiments are conducted on a collection of nonsmooth convex optimization problems taken from \cite{skajaaLimitedMemoryBFGS,SagastizabalCompositeProximalBundle2013,deoliveiraDoublyStabilizedBundle2016}.

The remainder of the paper is organized as follows. Section~2 describes the Adaptive GPB method in detail, including its OneCut and TwoCuts variants, as well as the main components of the subroutine and the main loop. Section~3 presents the complexity analysis of the algorithm, establishing bounds on the number of cycles and internal iterations required to reach an approximate solution. Section~4 reports the numerical experiments. Finally, Section~5 concludes the paper and discusses possible directions for future research.

\section{Adaptive GPB}\label{sec:adapgpb}

In this section, we present the proposed adaptive Generic Proximal Bundle (GPB) framework. We describe the adaptive update of the proximal parameter and introduce the one-cut and two-cuts variants considered in the analysis.

To solve problem \eqref{eq:ProbIntro2}, an iteration of GPB solves a proximal bundle subproblem of form
\begin{equation}\label{eq:x-pre}
    x = \underset{u\in \R^n}\argmin \left\{ \Gamma (u) + \frac{1}{2\lambda} \|u-x^c\|^2 \right\},
\end{equation}
where $\Gamma$ is a model for $f+h$
satisfying $\Gamma \in \mConv{n}$ and $\Gamma \leq \phi$, $\lambda$ is the step size, and
$x^c$ is the prox-center.
In what follows, we will denote by $\ell_f(\cdot;x)$ a linearization of the convex function $f$ at the point $x$, that is,
\begin{equation}\label{def:ell}
\ell_f(\cdot;x) := f(x)+\inner{f'(x)}{\cdot-x} \quad \forall x\in \dom h.
\end{equation}
Our adaptive variant of GPB has three ingredients described below:
\begin{itemize}
\item a model update blackbox, referred to as a bundle update (BU), that computes a new model 
$\Gamma^+$ for $f+h$ on the basis of a previous
model $\Gamma$, of a prox-center $x^c$, of a step size
$\lambda$, of a parameter $\tau \in (0,1)$, of an input $x$ satisfying \eqref{eq:x-pre},
and of a linearization $\ell_f(\cdot,x)$ of $f$ at $x$;
\item a subroutine that iterates performing inexact proximal steps (inexact in the sense that subproblems of form \eqref{eq:x-pre} are solved along iterations, whereas an exact proximal step would use $f+h$ instead of a model $\Gamma$ for $f+h$) until either we reach an $\bar \varepsilon/2$-approximate
proximal step in a sense quantified below
or we exit with failure at the end of a so-called bad cycle, without being able to compute such $\bar \varepsilon/2$-approximate proximal step;
\item a main loop that calls the subroutine and updates the step size $\lambda$ (either increasing or decreasing this step size) depending on the outputs of the subroutine.
\end{itemize}
We say that 
a sequence (or cycle)
of iterations of the
subroutine
is very good
when we "clearly" (as quantified by
parameters $\kappa_1 \leq 1 \leq \kappa_2$) observed the expected
$\tau$-contraction on the error of the approximate proximal step, for some
$\tau \in (0,1)$. In this case, the main loop will increase
the step size, multiplying the step size by 2.
When the cycle is bad (see a more precise definition below), we divide the
step size by 2.

\subsection{Bundle update blackbox}

We start by describing the bundle update (BU) blackbox.

\vspace{-0.3em}
\noindent\rule{\columnwidth}{0.8pt}
\vspace{-0.9em}

\noindent\textbf{Bundle Update (BU)}

\vspace{-0.7em}
\noindent\rule{\columnwidth}{0.8pt}
\vspace{-0.5em}

\noindent\textbf{Inputs:}
$(\lambda,\tau)\in\mathbb{R}_{++}\times(0,1)$ and
$(x^c,x,\Gamma)\in\mathbb{R}^n\times\mathbb{R}^n\times\mConv{n}$
such that $\Gamma\le\phi$ and \eqref{eq:x-pre} holds.

\vspace{0.1cm}

$\bullet$ Find function $\Gamma^+$ such that
\vspace{-0.9em}

{\small
\begin{equation}\label{def:Gamma}
\Gamma^+ \in \mConv{n},
\tau \bar\Gamma(\cdot)
+(1-\tau)[\ell_f(\cdot;x)+h(\cdot)]
\le
\Gamma^+(\cdot)
\le
\phi(\cdot).
\end{equation}
}
\vspace{-1.0em}

\noindent where $\ell_f(\cdot;\cdot)$ is as in \eqref{def:ell}
and $\bar\Gamma(\cdot)$ is such that
\vspace{-0.9em}

{\small
\begin{equation}\label{def:bar Gamma}
\bar\Gamma \in \mConv{n},
\bar\Gamma(x)=\Gamma(x),
x=
\underset{u\in\mathbb{R}^n}{\argmin}
\left\{
\bar\Gamma(u)
+
\frac{1}{2\lambda}\|u-x^c\|^2
\right\}.
\end{equation}
}
\vspace{-0.8cm}

\noindent\textbf{Output:} $\Gamma^+$.

\vspace{-0.4em}
\noindent\rule{\columnwidth}{0.8pt}

The above update scheme does not completely determine $\Gamma^+$
and	gives conditions suitable for the complexity analysis
	of Adaptive GPB.
A natural choice for the first model $\Gamma_1$ is $\Gamma_1 = h + \ell_{f}(,x_0)$ for a given first (arbitrary) trial point $x_0$.
The new model $\Gamma^{+}$ must be above a convex combination of
two functions, with $\tau \in (0,1)$ being the parameter defining this convex combination.
The first function in this convex combination is
$h$ plus
the last computed linearization $\ell_{f}(\cdot,x)$.
The second function in this convex combination is
another model
$\bar \Gamma$ satisfying some conditions (this model can be as a special case 
the previous model $\Gamma$).
All models provided by the blacbox BU are lower bounding 
convex functions
for the objective
$\phi$.

We now describe two particular BU blackboxes: a first
one where the model $\Gamma$ is made of one
linearization (one affine function) giving rise
to Adaptive GPB OneCut variant
and a second 
one where the model $\Gamma$ is made of two
linearizations (two affine functions) giving rise
to Adaptive GPB TwoCuts variant.\\

\par {\textbf{Adaptive GPB OneCut variant.}}
The bundle size remains constant equal to 1.
This scheme obtains $\Gamma^+$ as
\begin{equation}\label{eq:affine}
\Gamma^+= \Gamma^+_\tau := \tau \Gamma + (1-\tau) [\ell_f(\cdot;x)+h]
\end{equation}
where $x$ is as in \eqref{eq:x-pre} and $\tau\in (0,1)$. Clearly, if
$\Gamma$ is the sum of $h$ and
an affine function underneath $f$,
then so is $\Gamma^+$.\\

\par {\textbf{Adaptive GPB TwoCuts variant.}}
The bundle size remains constant equal to 2.
Assume that
$\Gamma=\max \{A_f, \allowbreak \ell_f(\cdot;x^-) \}+h$ where $A_f$ is an affine function satisfying $A_f\le f$ and $x^-$ is the previous iterate.
This scheme sets the next bundle function $\Gamma^+$ to one similar to
$\Gamma$ but with  $(x^-,A_f)$ replaced by
$(x,A_f^+)$ where
$A_f^+ =  \theta^+ A_f + (1-\theta^+) \ell_f(\cdot;x^-)$ for some
$\theta^+ \in[0,1]$ which does not
depend on $(\overline L_f,\overline M_f,\mu)$.\\

It is easy to check that these two variants of BU satisfy the requirements of this blackbox, see \cite{liang2023unified}
for details. When $h=0$
is the null function, it is also easy to see that with the OneCut variant  we have
analytic formulas for the optimal solution of the subproblems \eqref{eq:x-pre}
while for the TwoCuts variant we can obtain efficiently,
with line search and
iterations given in closed-form, an optimal solution of
\eqref{eq:x-pre}.
An advantage of these two variants is therefore that subproblems \eqref{eq:x-pre} are solved extremely fast, contrary to other bundle methods with different models $\Gamma$ which in general require a solver (quadratic when $h=0$) to solve \eqref{eq:x-pre}. 

More precisely, when $h=0$, for the OneCut variant, $\Gamma(u)$
is of form
$\langle \gamma,u \rangle + b$ and the solution of \eqref{eq:x-pre} is $x=x^c-\lambda \gamma$.

Now consider  the TwoCuts variant. 
We have 
$$
\Gamma = \max\{A_f, \ell_f(\cdot, x_-)\}
=\max_{\theta \in [0,1]} \theta A_f + (1-\theta)  \ell_f(\cdot, x_-),
$$ 
for
$A_f$ affine.
If $h = 0$ or if $h$ is the indicator of a box, then the function 
{\small
\begin{equation}
\label{eq:minmax_auxiliary}
\phi(\theta) = \min_{y \in \mathbb{R}^n} \theta A_f(y) + (1-\theta)  \ell_f(y, x_-)  + h(y) + \frac{1}{2\lambda} \lVert y - x^c \rVert^2,
\end{equation}}
can be computed by a simple closed formula and the minmax problem \eqref{eq:x-pre} can be reformulated as 
\begin{equation}
    \label{eq:bundle_subprob_linsearch}
    \max_{\theta \in [0,1]} \phi(\theta).
\end{equation}

For each $\theta \in [0,1]$, the solution $y^*(\theta)$ of problem \eqref{eq:minmax_auxiliary} can be computed through a closed formula (and its optimal value is $\phi(\theta)$).
The function $\phi$ is concave, from one dimensional variable $\theta \in [0,1]$
to $\mathbb{R}$, smooth, with a gradient which is affine in $y^*(\theta)$.
An optimal solution $\theta^*$ of
\eqref{eq:bundle_subprob_linsearch} can be computed by dichotomy, with corresponding optimal solution 
$y^*(\theta^*)$ of 
problem \eqref{eq:x-pre}.
When $h=0$, $\phi$ is a polynomial of degree 2 which can be maximized analytically
on $[0,1]$. For dichotomy, it is useful to use the derivative $\phi'(\theta)$
of $\phi$ which is given by
$$
\phi'(\theta)=A_f(y^*(\theta))-\ell_f(y^*(\theta), x_-).
$$
The dichotomy is then as follows: if $\phi'(0) \leq 0$ then $\theta^*=0$, if $\phi'(1) \geq 0$ then $\theta^*=1$, otherwise, do the following steps.\\
\par Step 0. Initialize $L=0$, $R=1$, $\theta=(L+R)/2$. Go to Step 1.\\
\par Step 1. If $|\phi'(\theta)| \leq \varepsilon$, then $\theta^*=\theta$ is an approximate optimal solution of \eqref{eq:bundle_subprob_linsearch}, stop. If $\phi'(\theta)>\varepsilon$, then $L=\theta$. If $\phi'(\theta)<-\varepsilon$, then $R=\theta$. Go to Step 2.\\
\par Step 2. $\theta=(L+R)/2$.
Go to Step 1.

\subsection{Adaptive GPB subroutine}
 
 We now give the pseudo-code for Adaptive GPB subroutine. It depends on two parameters $\kappa_1 \leq 1$ and $\kappa_2 \geq 1$.

\noindent\rule[0.5ex]{1\columnwidth}{1pt}

    Adaptive GPB subroutine with $\kappa_1 \leq 1 \leq \kappa_2$.

\noindent\rule[0.5ex]{1\columnwidth}{1pt}
\begin{itemize}
    \item [0.] {\textbf{Inputs:}} $\tau \in (0,1)$, $ \bar \varepsilon>0$, $\lambda>0$, and   
$x^c, \ty_0 \in \dom h $.
\par $\hspace*{1cm}$Set $i=1$.

    \item[1.] If $i=1$ then find $\Gamma_{1}$ such that
    \begin{equation}\label{ineq:require}
        \Gamma_{1}\in \mConv{n}, \quad \ell_f(\cdot;x^c)+h \le \Gamma_{1}\le  \phi;
    \end{equation}
    else, let $\Gamma_{i}$ be the output of the BU blackbox with input
    $(\lam,\tau,x^c)$ and $(x,\Gamma) = (\tx_{i-1},\Gamma_{i-1})$;
       \vspace{-0.2cm}
    \item[2.] Compute 
    \begin{equation}
        \tx_{i} =\underset{u\in \R^n}\argmin
        \left\lbrace \Gamma_{i}(u) +\frac{1}{2\lam}\|u- x^c \|^2 \right\rbrace,  \label{def:xj}
    \end{equation}
        choose $\ty_i \in \{\tx_i,\ty_{i-1}\} $ such that
    \begin{equation}\label{def:txj}
        \phi(\ty_{i}) = \min \left \lbrace\phi(\tx_{i}), \phi(\ty_{i-1}) \right\rbrace
    \end{equation}
    and set  
  {\small  \begin{equation}\label{ineq:hpe1}
        m_{i}=\Gamma_{i}(\tx_{i}) +\frac{1}{2\lam}\|\tx_{i}- x^c \|^2,  t_{i} = \phi(\ty_{i}) - m_{i}, 
    \alpha_i = \frac{t_i-\bar \varepsilon/4}{\tau^{i-1} (t_{1}-\bar \varepsilon/4)}; 
    \end{equation}}
\vspace{-0.8cm}
    \item[3.] If $t_i \le \bar \varepsilon/2$, then 
      \begin{itemize}
      \item
          declare the cycle to be {\textbf{good}} (resp., {\textbf{very good}}) if $\alpha_i > \kappa_1$
    (resp., $\alpha_i \le \kappa_1$);
    \item 
    {\textbf{stop}} the cycle 
    and output
    $(\tx_i,\ty_i)$;
      \end{itemize}
      else
    \begin{itemize}
    \item
        if  $\alpha_i > \kappa_2$ then {\textbf{stop}}
        the cycle and declare it to be {\bf bad};
    \end{itemize}
     \item[4.] 
     set $i\leftarrow
    i+1$ and go to step~1
\end{itemize}
\vspace{-0.4cm}
\noindent\rule[0.5ex]{1\columnwidth}{1pt}

Adaptive GPB subroutine
iteratively solves inexact proximal problem \eqref{def:xj}
and calls the model update
blackbox BU to update models
for $\phi$ on the basis of
the solutions to these inexact proximal problems and to linearizations of $f$ at the solutions to \eqref{def:xj}.
We show in Lemma \ref{lemk2b}  below that if
$\lambda \leq \lambda_*$ for $\lambda_*$
given in that lemma, then for all iterations $i$
of the subroutine we have
$t_i-{\bar \varepsilon}/4 \leq \tau(t_{i-1}-{\bar \varepsilon}/4)$ and therefore
$t_i-{\bar \varepsilon}/4 \leq \tau^{i-1}(t_1-{\bar \varepsilon}/4)$, or equivalently $\alpha_i \leq 1$.
When we are able to end the subroutine with
the stronger condition $\alpha_i \leq \kappa_1 \leq 1$ then we declare 
the cycle to be 
very good and will increase $\lambda$
(in the main loop of Adaptive GPB described below) for the next call to the subroutine. When, on the contrary, $\alpha_i >\kappa_2\geq 1$, we declare the cycle to be bad and will decrease $\lambda$ (in the main loop of Adaptive GPB described below)
for the next call to the subroutine.

\subsection{Main loop of Adaptive GPB}

The main loop of Adaptive GPB adapts the parameter $\lambda$ after each subroutine.
We call a cycle a sequence of iterations of the Adaptive GPB subroutine.

\noindent\rule[0.5ex]{1\columnwidth}{1pt}

Adaptive GPB - Main loop

\noindent\rule[0.5ex]{1\columnwidth}{1pt}
\begin{itemize}
    \item [0.] Let $\hat x_0=x_0 \in \dom h$, $\lambda=\lam_0>0$, $\bar \varepsilon >0$,
and $\tau \in (0,1)$ be given and set $k=1$ and $\hat y_0=\hat x_0$, let $\overline \lambda>0$;
    \item[1.] Call Adaptive GPB subroutine with input $(\tau,\bar \varepsilon,\lambda)$
    and $(x^c,\tilde y_0)=(\hat x_{k-1}, \hat y_{k-1})$ to perform the $k$-th cycle;
    
    \item[2.] If the $k$-th cycle is bad (resp., very good)  then set $\lam
    \leftarrow \lam/2$
    (resp., $\lam \leftarrow \min(2 \lam,\overline \lambda)$);
    \item[3.] If the $k$-th cycle is good or very good, then 
    \begin{itemize}
        \item 
         let  $(\hat x_k,\hat y_k)$ be its output, set 
    $k \leftarrow k+1$, and go to step 1;
    \end{itemize}
    else if the $k$-th cycle is bad then
    \begin{itemize}
        \item 
        go to step~1 to repeat the $k$-th cycle;
    \end{itemize}
\end{itemize}
\vspace{-0.4em}
\noindent\rule[0.5ex]{1\columnwidth}{1pt}

Adaptive GPB calls Adaptive GPB subroutine and updates parameter
$\lambda$ depending on the output of this subroutine: when the cycle is very good, we increase the step size for the next cycle: $\lam \leftarrow \min(2 \lam,\overline \lambda)$. When the cycle is bad, we halve the step size
for the next cycle: $\lam
    \leftarrow \lam/2$, otherwise 
    $\lambda$ is kept constant for the next cycle.
\paragraph{\textit{Remark on the adaptive parameters.}}
The ratio $\alpha_i$ quantifies how the observed decrease within a cycle compares to the nominal theoretical contraction. The parameters $\kappa_1$ and $\kappa_2$ provide a robust algorithmic interpretation of this measure, controlling the criteria for successful cycles and for triggering stepsize reductions. Different choices of these parameters affect the aggressiveness of the stepsize adaptation and the numerical behavior of the method. Their practical impact is investigated numerically in the Supplementary Material.

\section{Complexity analysis}\label{sec:compgpb}

This section is devoted to the convergence and iteration-complexity analysis of the proposed adaptive GPB method. We establish global convergence results and derive complexity bounds under both general convexity and strong convexity assumptions.

The goal of this section is to 
prove Theorem
\ref{thouterall} given below
which gives the 
complexity of Adaptive GPB.
We will use the recent results of the complexity of FSCO framework from
\cite{guigues2024universal}.
We define outer iterations of the method as the iterations
$k$ such that $t_k \leq \bar \varepsilon/2$ (final iterations of a cycle that is not bad).
It is easy to see that
we can cast the  outer iterations of the method
as the outer iterations of 
FSCO framework from
\cite{guigues2024universal}
with parameters 
$\chi=0$, $\sigma = \mu \underline{\lam}$, and $\varepsilon=\bar 
\varepsilon/2$.
However, the complexity of Adaptive GPB does not follow immediately from that of FSCO and
we provide in this section the additional proofs necessary to complete this complexity analysis.

The complexity analysis of FSCO gives us a control on the number of outer iterations
necessary to get an $\bar \varepsilon$ optimal solution. 
In our analysis, we will 
\begin{itemize}
\item[(i)] give a bound on the number of bad cycles;
\item[(ii)] derive a bound on the number of inner iterations of every cycle. 
\end{itemize}
Finally, we need to check that the assumptions needed to apply the results of the complexity analysis of FSCO are satisfied, namely that
\begin{itemize}
\item[(iii)] model $\Gamma \in \mConv{n}$;
\item[(iv)]  there is some (known) $\underline \lambda$
such that $\lambda \geq \underline \lambda$ for all 
iterations;
\item[(v)] the termination criterion
$t_k \leq \bar \varepsilon/2$
will be 
satisfied
in a cycle when $\lambda$ is sufficiently small.
\end{itemize}

Assumption (iii) is satisfied by Adaptive GPB.
We concentrate on showing (i), (ii), (iv), and (v) in what follows.
For our analysis, it will be convenient to define the
concept of stages.
Recall that a cycle is as a sequence of iterations of the subroutine.
Such cycle either is bad or ends for an iteration $i$ such that  $t_i \leq {\bar \varepsilon}/2$. 
Let us call a stage a sequence of
cycles that ends with a cycle ending with $t_i \leq {\bar \varepsilon}/2$, preceded by 
cycles (eventually none of them)  that are all bad. Eventually a stage
can be made of a single cycle that
ends with $t_i \leq {\bar \varepsilon}/2$, therefore preceded by no bad cycle. Using FSCO framework, we show that the number of stages is at most $T_1$ for $T_1$
given by (\ref{eq:T1}).
We will also show that
the number of consecutive bad cycles
is at most $k_2$ and therefore the maximal number of cycles in a stage is $T_2=1+k_2$ for 
$T_2$ and $k_2$ respectively given in
(\ref{eq:T2}) and \eqref{k2bound}. 
Finally, we will show that the number of iterations in a cycle is at most
$T_3:=k_3$ for $k_3$ given in \eqref{formk3}. With these results at hand, we will conclude that the functional complexity
to get an $\bar \varepsilon$ optimal solution with Adaptive GPB
is at most $T_1 T_2 T_2$, see Theorem \ref{thouterall}. We now proceed with our work plan.

Let $x_*$ denote the closest solution of \eqref{eq:ProbIntro2} to the initial point $x_0$ of Adaptive GPB
and let $d_0$ denote its distance to $x_0$, i.e.,
{\small
\begin{equation}\label{defd0}
\|x_0-x_*\|= \min \{\|x-x_0\| : x \in X^*\}, \quad d_0=\|x_0 - x_*\|. 
\end{equation}}

We first show
in Lemma 
\ref{lemk2b} (whose proof is deferred to the Appendix)
that 
the assumptions of the framework are satisfied:
there is
$\underline \lambda>0$ such that
$\lambda \geq \underline \lambda$
for every value $\lambda$ of the step size in the algorithm.
We also show that if we start a cycle with $\lambda$
sufficiently small then
the cycle is not bad, i.e., we end up
the cycle with $t_i \leq {\bar \varepsilon}/2$.

In what follows, we denote by $j_k+1$ the first iteration of cycle $k$.

\begin{lemma} \label{lemk2b}
Define $\varepsilon=\bar \varepsilon/2$ and $\lambda_*$ by
$$
\tau =  1- \left(1+ \frac{4\lam_*(2 {\overline M}_f^2 + \varepsilon {\overline L}_f)  }{\varepsilon}\right)^{-1} \text{, i.e.},
$$   

$$
\lambda_*=\left(\frac{\tau}{1-\tau}\right)\frac{\varepsilon}{4(2 \overline M_f^2+\varepsilon \overline L_f)}.
$$
Then the following holds for Adaptive GPB:
\begin{itemize}
\item[(a)]  if a cycle starts with $\lambda \leq \lambda_*$ then the cycle is not bad and ends  in at most $k_1$ iterations where
{\small
\begin{equation}\label{defk1}
k_1 = 1+\left \lceil k_0 
\right
\rceil
\mbox{ for }
k_0=\frac{1}{\log(\tau^{-1})}\log\left(
\frac{{\bar t}-{\bar \varepsilon}/4}{{\bar \varepsilon}/4}\right)
\end{equation}}
with
\begin{equation}\label{deftbar}
\bar t = 
{\overline M}_f^2 +
4(2+{\overline L}_f)(d_0\max(1,2{\overline \lambda}\,{\overline L}_f)+{\overline \lambda}\,{\overline M}_f)^2;
\end{equation}
\item[(b)] step size $\lambda$ is always
greater than or equal to
\begin{equation}\label{deflbar}
\underline \lambda=\frac{\lambda_*}{2}=
\left(\frac{\tau}{1-\tau}\right)\frac{\varepsilon}{8(2 \overline M_f^2+\varepsilon \overline L_f)}
\end{equation}
and if $\overline \lambda>\underline \lambda$ the number of consecutive bad cycles is at most
\begin{equation}\label{k2bound}
k_2 = \left \lceil \frac{1}{\log(2)} \log\left( \frac{{\overline \lambda}}{\underline \lambda}  \right)  \right \rceil.
\end{equation}
\end{itemize}
\end{lemma}

\paragraph{\textit{Remark on the bounds $\underline{\lambda}$ and $\bar{\lambda}$.}}
The bounds $\underline{\lambda}$ and $\bar{\lambda}$ play different roles in the analysis. Parameter $\bar{\lambda}$ is the algorithmic upper bound in the adaptive update $\lambda \leftarrow \min(2\lambda,\bar{\lambda})$, preventing the prox-stepsize sequence from growing indefinitely. The lower bound $\underline{\lambda}$, defined in (\ref{deflbar}), characterizes the stepsize regime in which contraction and complexity estimates hold and depends on problem constants. The term $\log\!\left(\bar{\lambda}/\underline{\lambda}\right)$ therefore quantifies the worst-case number of stepsize reductions required to bring $\lambda$ into the contraction regime.

\par To study the outer iterations of Adaptive GPB ending with $t_i \leq {\bar \varepsilon}/2$, we need the following result, where items a) and c) derive from Lemma A.2 of \cite{guigues2024universal} and item b) can be easily justified using a) and the fact that $\eta_j \geq 0$.
\begin{lemma} \label{lm:easyrecur1}
		Assume that sequences $\{\gamma_j\}$, $\{\eta_j\}$, and  $\{\alpha_j\}$ satisfy for every $j\ge 1$, $\gamma_j \ge {\underline \gamma}$ and 
		\beq \label{eq:easyrecur1}
		\gamma_j \eta_j \le \alpha_{j-1} - (1+\sigma)\alpha_j + \gamma_j \delta
		\eeq
  for some $\sigma \ge 0$, $\delta \ge 0$ and ${\underline \gamma} > 0$. Then, the following statements hold:
	\begin{itemize}
            \item[a)] for every $k\ge 1$,
            \begin{equation}\label{ineq:eta}
            \min_{1\le j \le k}  \eta_j \le \frac{\alpha_0 - (1+\sigma)^k \alpha_k}{\sum_{j=1}^k (1+\sigma)^{j-1} \gamma_j} + \delta;
        \end{equation}
            \item[b)] if the sequence $\{\eta_j\}$ is nonnegative, then for every $k\ge 1$,
            \begin{equation}\label{ineq:alpha}
            \alpha_k \le \frac{\alpha_0}{(1+\sigma)^k}  +
		\frac{\sum_{j=1}^k (1+\sigma)^{j-1}\gamma_j \delta}{(1+\sigma)^k};
        \end{equation}
		\item[c)]  if the sequence $\{\alpha_j\}$ is nonnegative, then $ \min_{1\le j \le k} \eta_j \le 2\delta $ for every $ k\ge 1$ such that
		\[
		k \ge \min \left\lbrace  \frac{1+\sigma}{\sigma} \log\left( \frac{\sigma \alpha_0}{{\underline \gamma}\delta} + 1 \right), \frac{\alpha_0}{{\underline \gamma}\delta} \right\rbrace
		\]
		with the convention that the first term is equal to the second term
		when $\sigma=0$. (Note  that the first term converges to the second term
		as $\sigma \downarrow 0$.)
	\end{itemize}
\end{lemma}
Using the notation in \cite{guigues2024universal}, considering a prox-center $x^{-} \in \operatorname{dom} \phi$ and $u \in \mathbb{R}^n$, each iteration of FSCO solves approximately a prox subproblem of the form:
\begin{equation}\label{prox_sub_bb}
\min _{u \in \mathbb{R}^n}\left\{\phi(u)+\frac{1}{2 \lambda}\left\|u-x^{-}\right\|^2\right\},
\end{equation}
for some $\lambda>0$, to obtain an inexact solution $y$. The following black box BB is used by FSCO and describes how (\ref{prox_sub_bb}) is approximately solved. Given a prox-center $x^-$, an accuracy $\varepsilon$, BB outputs new approximate solutions $y$ and $x$ (with some gap for solution $y$ upper-bounded by $\varepsilon$) and a new model $\Gamma$ for $\phi$. 

\noindent\rule[0.5ex]{1\columnwidth}{1pt}

Black-box (BB) $\left(x^{-},\varepsilon\right)$

\noindent\rule[0.5ex]{1\columnwidth}{1pt}
\textbf{Input:} $\left(x^{-}, \varepsilon\right) \in \operatorname{dom} \phi \times \mathbb{R}_{++}$.
\\
\textbf{Output:} $(x, y, \Gamma) \in \operatorname{dom} \phi \times \operatorname{dom} \phi \times \overline{\operatorname{Conv}}\left(\mathbb{R}^n\right) \times \mathbb{R}_{++}$ satisfying $\Gamma \leq \phi$,
$$
\phi(y)+\frac{\chi}{2 \lambda}\left\|y-x^{-}\right\|^2-\min _{u \in \mathbb{R}^n}\left\{\Gamma(u)+\frac{1}{2 \lambda}\left\|u-x^{-}\right\|^2\right\} \leq \varepsilon, 
$$$$
\text{ and }
x=\underset{u \in \mathbb{R}^n}{\operatorname{argmin}}\left\{\Gamma(u)+\frac{1}{2 \lambda}\left\|u-x^{-}\right\|^2\right\}.
$$
\noindent\rule[0.5ex]{1\columnwidth}{1pt}

The following proposition is necessary for the proof of Theorem \ref{thouter}. This is Proposition 3.1 of \cite{guigues2024universal}, considering $\chi = 0$.

\begin{proposition}\label{prop3.1_glm24}
Suppose that $\phi \in \overline{\operatorname{Conv}}\left(\mathbb{R}^n\right)$ and let $\mu=\mu_\phi$. Suppose also that $(x, y, \Gamma)=\operatorname{BB}\left(x^{-}, \varepsilon\right)$ for some $\left(x^{-}, \varepsilon\right) \in \operatorname{dom} \phi \times \mathbb{R}_{++}^n$, and $\Gamma \in \overline{\operatorname{Conv}}_\nu\left(\mathbb{R}^n\right)$ for some $\nu \in[0, \mu]$. Then, for all $u \in \mathbb{R}^n$, we have, with $\sigma:={\underline \lambda} \nu$:
\begin{equation}
2 \lambda[\phi(y)-\phi(u)] \leq 2 \lambda \varepsilon +\left\|x^{-}-u\right\|^2-(1+\sigma)\|x-u\|^2.
\end{equation}
\end{proposition}

The complexity of the outer iterations
of Adaptive GPB (the maximal number of cycles generating a solution with
$t_i \leq \bar \varepsilon/2$) is given in the following theorem.
\begin{theorem}\label{thouter}
For a given tolerance $\bar \varepsilon>0$, 
 the number of  outer iterations $k$ 
 such that
 $t_k \leq \bar \varepsilon$
  necessary to generate an iterate $\hat y_k$ satisfying $\phi(\hat y_k)-\phi_* \leq \bar \varepsilon$ is at most
\begin{equation}\label{complprimal}
\min\left\{\left( 1+\frac{1}{\underline{\lam} \mu}\right) \log\left(1+ \frac{\mu d_0^2}{\bar \varepsilon}\right), \frac{d_0^2}{{\underline \lambda}{\bar \varepsilon}} \right\}.
\end{equation}
\end{theorem}

\begin{proof}
Recall that 
we can cast the final iteration of
cycles 
of Adaptive GPB that 
are not bad
as the outer iterations of 
FSCO framework from
\cite{guigues2024universal}
with parameters 
$\nu=\mu$, $\sigma = \mu \underline{\lam}$, $\varepsilon=\bar \varepsilon/2$
and using Proposition \ref{prop3.1_glm24} with
these parameters, we have
{\small 
\begin{equation}\label{ohiustb}
    2 \lambda_k\left[\phi(\hat y_k)-\phi(u)\right]  \le 2\lam_k \varepsilon+ \left\|\hat x_{k-1}-u\right\|^2 - (1+\mu \underline{\lam})\|\hat x_k-u\|^2
\end{equation}}
for every $u$
where $\lambda_k$
is the value of
$\lambda$
for cycle $k$
of Adaptive GPB.
It is easy to see that \eqref{ohiustb} with $u=x_*$, $\varepsilon=\bar \varepsilon/2$, satisfies \eqref{eq:easyrecur1} with $\sigma={\underline \lam} \mu$ and
{\small
\begin{equation}\label{eq:parameter}
    \gamma_k = 2\lam_k, \quad \eta_k=\phi(\hat y_k) - \phi_*, \quad \alpha_k = \|\hat x_k-x_*\|^2, \quad \delta = \frac{\bar \varepsilon}2.
\end{equation}}
Also, note that
$
{\underline \gamma} = 2{\underline \lam}$. It follows from Lemma \ref{lm:easyrecur1}(c) with the above parameters that the complexity to find a $\bar \varepsilon$-solution is
\begin{equation}\label{eq:complexity}
    \min\left\{ \left(1+\frac{1}{\mu \underline{\lam}}\right) \log\left(1+ \frac{\sigma d_0^2}{{\underline \lambda}{\bar \varepsilon}}\right), \frac{d_0^2}{{\underline \lambda}{\bar \varepsilon}} \right\}.  
\end{equation}
\end{proof}

\begin{proposition}\label{propmaxgpb} The maximal number of iterations of Adaptive GPB subroutine (called inner iterations of Adpative GPB), i.e., the maximal number of iterations of every cycle, is 
\begin{equation}\label{formk3}
k_3=2+\left\lceil
\frac{1}{\ln(\tau^{-1})}
\ln\left(\frac{\kappa_2(\bar t-{{\bar \varepsilon}/4})}{{\bar \varepsilon}/4} \right)
\right \rceil.
\end{equation}
where 
$\bar t$ is given by \eqref{deftbar}.
\end{proposition}
\begin{proof}
Recall from \eqref{upptbar} that
$t_1 \leq \bar t$ which implies
$$
k_3 > 1+ \frac{1}{\ln(\tau^{-1})}
\ln\left(\frac{\kappa_2(t_1-{{\bar \varepsilon}/4})}{{\bar \varepsilon}/4}\right)
$$
which can be written
\begin{equation}\label{ineqtaui0}
\tau^{k_3-1} < \frac{{\bar \varepsilon}/4}{\kappa_2(t_1-{\bar \varepsilon}/4)}.
\end{equation}
Let us now show by contradiction that
every cycle of Adaptive GPB subroutine terminates in at most $k_3$ iterations, 
with either $t_i \leq {\bar \varepsilon}/2$
or with $\alpha_i>\kappa_2$ (bad cycles). Assume for contradiction that
$t_i> \bar \varepsilon/2$
and $\alpha_i \le \kappa_2$ for every $i =1,\ldots,k_3$. Then
\[
     \kappa_2 \ge \alpha_{k_3} = \frac{t_{k_3}-\bar \varepsilon/4}{\tau^{k_3-1} (t_{1}-\bar \varepsilon/4)} >
     \frac{\bar \varepsilon/4}{\tau^{k_3-1} (t_{1}-\bar \varepsilon/4)}
    \]
which is in contradiction with \eqref{ineqtaui0}.
\end{proof}

The final complexity of Adaptive GPB, giving the maximal number of inner iterations (total number of iterations) to get an $\bar \varepsilon$-optimal solution of \eqref{eq:ProbIntro2} is given by Theorem \ref{thouterall}.
\begin{theorem}\label{thouterall}
Take $\overline \lambda>\underline \lambda$. The maximal total number of inner iterations
for Adaptive GPB to find a solution
$\hat y_k$ satisfying
$\phi(\hat y_k)-\phi_* \leq {\bar \varepsilon}$ is equal to $T_1T_2T_3$, where

\small{
\begin{align}
T_1 &=
\min\left\{
\left(1+\frac{1}{\mu \underline{\lambda}}\right)
\log\left(1+\frac{\mu d_0^2}{\bar{\varepsilon}}\right),
\frac{d_0^2}{\underline{\lambda}\bar{\varepsilon}}
\right\}, \label{eq:T1}\\
T_2 &=
1+\left\lceil
\frac{1}{\log(2)}
\log\left(
\frac{\overline{\lambda}}{\underline{\lambda}}
\right)
\right\rceil, \label{eq:T2}\\
T_3 &=
2+\left\lceil
\frac{1}{\ln(\tau^{-1})}
\ln\left(
\frac{\kappa_2(\bar t-\bar{\varepsilon}/4)}
{\bar{\varepsilon}/4}
\right)
\right\rceil. \label{eq:T3}
\end{align}}

where $\underline \lambda$, $d_0$, and 
$\bar t$ are respectively given in
\eqref{deflbar}, \eqref{defd0}, and \eqref{deftbar}.
\end{theorem}
\begin{proof}
To bound the number of iterations to get an an $\bar \varepsilon$ optimal solution of \eqref{eq:ProbIntro2}, recall that we have grouped the inner iterations in cycles and the cycles in stages. Therefore if $T_1$ is
the maximal number of stages, if 
$T_2$ is the maximal number of cycles in a stage, and if $T_3$ is the maximal number of iterations in a cycle, the complexity of Adaptive GPB is 
$T_1 T_2 T_3$. By Theorem \ref{thouter}, $T_1$ is given by (\ref{eq:T1}). We also have that $T_2$
is 1 plus the maximal number of consecutive bad cycles.
By Lemma \ref{lemk2b}-\eqref{k2bound}
we have that  this quantity $T_2$
is given by (\ref{eq:T2}).
Finally, by Proposition \ref{propmaxgpb}, $T_3$ is given by (\ref{eq:T3}).

This achieves the proof of the theorem.
\end{proof}

The bound in Theorem \ref{thouterall} recovers the classical complexity rates for first-order methods: a logarithmic dependence on $\bar{\varepsilon}$ in the strongly convex case and a $\mathcal{O}(1/\bar{\varepsilon})$ dependence in the general convex setting. The remaining factors depend only logarithmically on the algorithmic parameters. Hence, the bound is optimal up to logarithmic factors.

The iteration-complexity bound established in Theorem \ref{thouterall} for the Adaptive GPB method can be directly compared with results obtained for related proximal bundle schemes. In particular, Theorem 3.1 in \cite{liang2023unified} derived a unified complexity bound for the Generic Proximal Bundle (GPB) framework, showing that any GPB variant achieves an $\varepsilon$-solution in a number of iterations bounded (up to logarithmic factors) by expressions depending on $(M_f, L_f, \mu, d_0, \varepsilon)$. Similarly, in \cite{guigues2024universal}, the Theorem 2.3 established functional iteration-complexity bounds for the Universal Proximal Bundle (U-PB) method, which under mild conditions reduce to nearly optimal rates of order $\bar{\mathcal{O}}\left(M_f^2 /(\varepsilon \mu_\phi) + L_f / \mu_\phi\right)$.
Compared with these results, Theorem \ref{thouterall} shows that Adaptive GPB attains the same classical first-order dependence on the tolerance $\varepsilon$, the strong convexity parameter $\mu$, and the initial distance $d_0$, up to logarithmic factors. In contrast to \cite{liang2023unified}, where suitable parameters must be chosen using knowledge of problem constants, the adaptive scheme adjusts the stepsize dynamically through a cycle-based contraction test. As a consequence, the potential overhead of adaptivity appears only in logarithmic terms associated with stepsize adjustment, while preserving the same fundamental $\varepsilon$-dependence. Therefore, the rates obtained here are consistent with, and complementary to, the bounds previously established for GPB in \cite{liang2023unified} and for U-PB in \cite{guigues2024universal}, confirming that Adaptive GPB fits naturally within the broader bundle-method literature.

\if{
\subsection{Inner convergence}
It follows from Lemma 3.9 of \cite{guigues2024universal} that
\begin{equation}\label{ineq:tj-recur}
         t_{j+1}-\frac{\varepsilon}2 \le \tau \left(t_j -\frac{\varepsilon}2\right)
     \end{equation}
with
\begin{equation}\label{deftauk}
\tau \geq  1- \left(1+ \frac{4\lam({\overline M}L_f^2 + \varepsilon {\overline L}_f)  }{\varepsilon}\right)^{-1}.
\end{equation} 

So, in either subroutine, we want to revise $\lam$ or $\tau$ so that
\eqref{deftauk} holds.

Consider Subroutine 2 and let us call a cycle a sequence
of iterations, called inner iterations, in this subroutine.
Oberve that for this subroutine, we end either with success or with failure.
We have bounded from above the number of cycles that end with success.
Let us bound for these
cycles the number of inner iterations.
We have two types
of such cycles: 
\begin{itemize}
\item[(a)] Cycles where
for all iteration
$k$ before termination
at an iteration $i$ 
with $t_i \leq {\bar \varepsilon}/2$ we have
\begin{equation}\label{cyclea}
t_k-{\bar \varepsilon}/4 \leq \tau(t_{k-1}-{\bar \varepsilon}/4),\;k=1,\ldots,i;
\end{equation}
\item[(b)] cycles where
for some iteration
$k$ before termination
at an iteration $i$ 
with $t_i \leq {\bar \varepsilon}/2$ we have
$t_k-{\bar \varepsilon}/4 > \tau(t_{k-1}-{\bar \varepsilon}/4)$. In this cycle, we have 
at least an update
of $\tau$
of form
$\tau \leftarrow (1+\tau)/2$
and at most ${\bar n}_{\tau}$  such updates.
\end{itemize} 

For cycles described in item
(a) above, we have
at most
\begin{equation}\label{defkob}
\bar k=1+\left \lceil k_0 
\right
\rceil
\mbox{ for }
k_0=\frac{1}{\log(\tau_0)}\log\left(
\frac{{\bar \varepsilon}/4}{{\bar t}-{\bar \varepsilon}/4}\right)
\geq {\bar k}_0= \frac{1}{\log(\tau)}\log\left(
\frac{{\bar \varepsilon}/4}{{\bar t}-{\bar \varepsilon}/4}\right)
\end{equation}
inner iterations. Indeed,
for $k \geq {\bar k}$ we have 
\begin{equation}\label{successkbar}
t_{k} \leq {\bar \varepsilon}/2
\end{equation}
which implies that we finish the cycle with success in at most 
$\bar k$ inner iterations.
Indeed, relation \eqref{successkbar} comes
from the relations
$$
\begin{array}{lcl}
t_{k} &\leq &{\bar \varepsilon}/4 + \tau^{{k}-1}(t_{1}-{\bar \varepsilon}/4) \mbox{ by }\eqref{cyclea},\\
&\leq &{\bar \varepsilon}/4 + \tau^{{k}-1}({\bar t}-{\bar \varepsilon}/4) \mbox{ using }t_1 \leq {\bar t},\\
&\leq &{\bar \varepsilon}/4 + \tau^{{\bar k}-1}({\bar t}-{\bar \varepsilon}/4) \mbox{ for }k \geq {\bar k},\\
&\leq &{\bar \varepsilon}/4 + \tau^{k_0}({\bar t}-{\bar \varepsilon}/4),\\
&\leq &{\bar \varepsilon}/4 + \tau^{\bar k_0}({\bar t}-{\bar \varepsilon}/4),\\
&\leq &{\bar \varepsilon}/2
\mbox{ using }
\eqref{defkob},
\end{array}
$$
for $k \geq \bar k$.

For cycles of type (b) above, 
the maximal number of iterations
$k$ such that
$t_k-{\bar \varepsilon}/4 > \tau(t_{k-1}-{\bar \varepsilon}/4)$ is $\bar n_{\tau}$
and in the worst case, we have
$\bar k-1$ iterations between such iterations
(indeed if we have $\bar k$ iterations
$i$ with $t_i-{\bar \varepsilon}/4 \leq \tau(t_{i-1}-{\bar \varepsilon}/4)$ then we finish the cycle with success) so the maximal number of inner iterations
for cycles of type (b) above is
$\bar k(\bar n_{\tau}+1)$.
Indeed, finishing with success, we have
at most $\bar n_\tau$
sequences of iterations 
made up of at most
$\bar k$ iterations 
$i$ with
$t_k-{\bar \varepsilon}/4 > \tau(t_{k-1}-{\bar \varepsilon}/4)$
followed by an iteration $k$ with 
$t_k-{\bar \varepsilon}/4 > \tau(t_{k-1}-{\bar \varepsilon}/4)$

Therefore
the number of inner iterations
for cycles with success is at most
$$
(1+{\bar n}_{\tau})
\left(1+ \left \lceil \frac{1}{\log(\tau_0)}\log\left(
\frac{{\bar \varepsilon}/4}{{\bar t}-{\bar \varepsilon}/4}\right)
\right \rceil\right).
$$

For cycles that end with failure, using the same reasoning as above, we have at most
$(\bar n_{\tau}+1)\bar k$ inner iterations and the number of cycles with failure
is upper bounded by 
$\ell_0$,
the smallest integer $\ell$
such that
$$
\frac{\tau_0}{2^\ell}
+\sum_{i=1}^{\ell} \frac{1}{2^i} \geq 1 - \frac{1}{1+\frac{\lambda_0 C}{2^{\ell}}},
$$
where
$$
C=\frac{4({\overline M}_f^2 + \varepsilon {\overline L}_f)  }{\varepsilon}.
$$
Indeed, after $\ell$ cycles
with failure, we have
updated at least $\ell$ times
$\tau$ as $\tau\leftarrow (1+\tau)/2$ 
and $\lambda$ as
$\lambda \leftarrow \lambda/2$.
Therefore $\tau$ is equal to
$\frac{\tau_0}{2^\ell}
+\sum_{i=1}^{\ell} \frac{1}{2^i}$
while $\lambda$ is equal
to $\lambda_0/2^\ell$.
It follows that for $\ell \geq \ell_0$
relation \eqref{deftauk} holds, meaning
that we will only have cycles ending with success.

Therefore the complexity of GPB with subroutine 2
is
$$
\begin{array}{l}
\Big(1+\bar n_\tau\Big)\left(\ell_0+\min\left\{\left( 1+\frac{1}{\underline{\lam} \mu}\right) \log\left(1+ \frac{\mu d_0^2}{\bar \varepsilon}\right), \frac{d_0^2}{{\underline \lambda}{\bar \varepsilon}} \right\}\right)
\left(1+\left \lceil \frac{1}{\log(\tau_0)}\log\left(
\frac{{\bar \varepsilon}/4}{{\bar t}-{\bar \varepsilon}/4}\right)
\right \rceil
\right)
\end{array}
$$

Now consider Subroutine 1.
When we end with good success, we have $i=1$ or $\alpha_i \leq \kappa_1$
which implies, using the same reasoning as before, that the number of inner iterations of such cycles
is at most
$$
1+\left \lceil 
\frac{1}{\log(\tau_0)}\log\left(
\frac{{\bar \varepsilon}/4}{\kappa_1({\bar t}-{\bar \varepsilon}/4}\right)
\right
\rceil.
$$
Similarly, in Subroutine 1, for cycles ending neither with good
success nor with failure we have $\alpha_i \leq \kappa_2$ and following
the reasoning above, at most
$$
1+\left \lceil 
\frac{1}{\log(\tau_0)}\log\left(
\frac{{\bar \varepsilon}/4}{\kappa_2({\bar t}-{\bar \varepsilon}/4}\right)
\right
\rceil
$$
innner iterations.

}\fi

\section{Numerical experiments}\label{sec:num}
This section reports numerical experiments illustrating the practical performance of the proposed methods.

We present several academic problems frequently used in the literature of Bundle
Methods and present in \cite{
skajaaLimitedMemoryBFGS,
SagastizabalCompositeProximalBundle2013, deoliveiraDoublyStabilizedBundle2016}.
The problems are solved until an approximate solution gives a value which is
$\varepsilon > 0$ close to the known optimal value. We report both the CPU time
and the number of oracle calls. The experiments compare the following methods:
LVL \cite{LemarechalNewVariantsBundle1995},
PBM-1 \cite{kiwielProximityControlBundle1990, deoliveiraDoublyStabilizedBundle2016},
U-PB \cite{guigues2024universal},
Adaptive GPB OneCut and Adaptive GPB TwoCuts,
and GPB variants (GPB 1C, GPB 2C, and GPB MC) from \cite{liang2023unified}.

All methods described above were tested in every scenario. However, only the most competitive results are reported in the tables. In particular, when a method required a significantly longer runtime (by an order of magnitude) or failed to converge within the maximum number of iterations, its result is omitted and indicated by “--” (we run the process considering a maximum number of iterations well above the average number of iterations that methods execute to reach the desired tolerance). Additional implementation details, detailed numerical tables and figures for all experiments and parameters adopted for each experiment are presented in the online supplementary material.

\subsection{BadGuy}\label{secbadguy}
The following small scale problem is denoted as BadGuy in
\cite{SagastizabalCompositeProximalBundle2013} and was introduced in \cite[Example
1.1.2 p.277]{Hiriart-UrrutyConvexAnalysisMinimization1993} as a pathological
example where the cutting plane method needs roughly
$\left(\frac{C}{2\sqrt{\varepsilon}}\right)^n$ iterations to converge.
Bundle Methods were designed to be regularization of the cutting plane method
and the problem BadGuy serves as an illustration of the process of bundle
methods to handle such pathological examples where the iterates of the cutting
plane method suffer from a so-called "bang-bang" effect where they are thrown
from one side of the domain to the other.

Given $\varepsilon \in (0, 1/2)$, we want to minimize the objective function of
BadGuy for $n=10$ 
(as in the original 
formulation of the
problem): 
\begin{equation}
    \label{eq:BadGuy}
    f_1(y, \eta) = \max \left\{ \lvert \eta \rvert, -1 + 2\varepsilon + \lVert y \rVert_2 \right\},
\end{equation}
where $(y,\eta)$ is in the unit ball of $\R^n \times \R$. The optimal value is
$0$ attained at $(0, y)$ for any $y \in B(0, 1 - 2\varepsilon)$, considering the starting point $x_0=(1,\cdots,1).$

Our Adaptive GPB variants solve the problem much quicker even if the tolerance is lower. 

\subsection{Chained CB3 II}\label{seccb3}

We consider 
Chained CB3 II problem
from \cite{CharalambousNonlinearMinimaxOptimization1976} as reported in
\cite{HaaralaNewLimitedMemory2004}. In this problem we aim to minimize the
piecewise convex function
{\small
\begin{equation}
    \label{eq:CB3}
    \begin{aligned}
    f_2(x) = \max \left\{ \sum_{i=1}^{n-1}(x_i^4 + x_{i+1}^2), \sum_{i=1}^{n-1} (2-x_i)^2 + (2-x_{i+1})^2, \right. \\
 \left. \sum_{i=1}^{n-1} 2\exp(-x_i + x_{i+1})\right\},
    \end{aligned}
\end{equation}}
where $x\in \R^n$. The optimal value is equal to $2(n-1)$. Since the optimal target values are in the range of $10^3$--$10^4$, this tolerance already guarantees a good relative accuracy for the problem. Among the tested methods, Adaptive GPB OneCut achieved the best performance in terms of CPU time for both problem sizes. Although Adaptive GPB TwoCuts was not the fastest method overall, it obtained the second best result for $n=1000$ and remained competitive for $n=5000$, not appearing among the slowest methods. These results highlight the efficiency of the adaptive variants, particularly the OneCut version, for this problem instance. We considering $x_0=(0,\cdots,0)$ for all tests and a maximum number of iterations of $500000$.

\subsection{MaxQuad}\label{maxquad_subsection}
We consider the famous (to test methods for nondifferentiable problems) convex non-smooth optimization problem introduced by Lemaréchal
et al. and called MaxQuad in \cite[Section
12.1.1]{bonnansNumericalOptimizationTheoretical2006}. The problem is to minimize over $\R^n$ 
\begin{equation}
f_3(x) = \max_{1\leq \ell \leq N} \langle A_\ell x, x \rangle + \langle b_\ell,x \rangle,
\end{equation}
with $n=10$, $N=5$, and where for each $\ell \in \{1, \ldots, 5\}$ we have
$b_\ell(i) = -\exp (i/ \ell) \sin (i \ell)$, $i=1,\ldots,n$, and $A_{\ell}$ is
symmetric, considering $i=1,\ldots,n,k=i+1,\ldots,n$, with 
\begin{equation}
    \left\{
        \begin{array}{lcl}
            A_{\ell}(i,k) &= &\exp(i/k)\cos(i k)\sin(\ell),\\
        A_{\ell}(i,i)&=&\displaystyle (i/n)\left|\sin(\ell)\right|+\sum_{k \neq i} |A_{\ell}(i,k)|. 
        \end{array}
    \right.
\end{equation}

One of the main features of MaxQuad is that at the unique minimizer, 4 of the 5
quadratics defining $f_3$ are active. Another important feature of MaxQuad is
that each quadratics is flat in some directions, steep on others and overall,
the pointwise maximum $f_3$ is steep in every direction. Based on the results, it can be seen that Adaptive GPB OneCut delivers similar time results to GPB 1C, only losing in the case where it is considered $\varepsilon=10^{-3}$. Adaptive GPB TwoCuts shows improved performance for all tolerances. We considering $x_0=(1,\cdots,1)$ for all tests and a maximum number of iterations of $500000$.

\subsection{MXHILB}\label{secMXHILB}
The following problem is associated with the computation of the kernel of the
Hilbert matrix. It was first proposed by Kiwiel in
\cite{KiwielEllipsoidTrustRegion1989} and is given by
\begin{equation}
    \label{eq:MXHILB}
 \min_{x \in \mathbb{R}^n} \; f_4(x):= \max_{1\leq i \leq n}\left\{ \lvert x_i \rvert \sum_{j=1}^n \frac{1}{i+j-1} \right\}.
\end{equation}
The minimizer is $x^* = 0$ with optimal value $0$. We start from $x_0=(1,\cdots,1)$ $\in \R^n$. 

We tested different values of $n$ for the MXHILB problem, considering a maximum of $300000$ iterations. 
Starting from the first tested dimension ($n=100$), GPB OneCut and TwoCuts frequently reached the maximum 
number of iterations without achieving the desired tolerance, despite having relatively fast individual iterations. 
The U-PB method exhibited a similar issue from $n=500$ onward, but without the same per-iteration speed. 
Although GPB MultiCuts converged within the iteration limit, its execution time was significantly higher, 
remaining far behind the other methods. For these reasons, the analysis in this subsection focuses only on the methods with competitive performance, 
namely PBM-1, LVL, and Adaptive GPB (OneCut and TwoCuts). The performance of both versions of the adaptive GPB is superior to LVL and PBM-1, as well as superior to other methods discarded for the aforementioned reasoning.

\subsection{RandMaxQuad}\label{sec_randmaxquad}
We now consider the box constrained and non-smooth optimization problem
\begin{equation}
    \label{eq:RandMaxQuad}
\min_{{\underline x} \leq x \leq {\overline x}   }\;    f_5(x) := \left[ \max_{1\leq i\leq N} \langle A_ix, x \rangle + \langle b_i,x \rangle + \alpha \lVert x \rVert_1 \right].
\end{equation}

This problem class was used for numerical tests in
\cite{deoliveiraDoublyStabilizedBundle2016}. We take $\alpha = 0.5$, $n = 200$, $N = 5$, and $A_i$ is a randomly generated symmetric positive
definite matrix with condition number equal to $n$ and $b_i$ is a random vector.
The initial point is $(1, 1,\ldots, 1)$ for all methods.

Our Adaptive GPB TwoCuts variant is faster than the other models tested. However, the Adaptive GPB OneCut model does not perform as well in this example, despite outperforming U-PB, LVL, PBM-1, and GPB MC.

\subsection{TiltedNorm}\label{sectilted}

We want to minimize over $X$ the function
\begin{equation}
    \label{eq:tiltednorm}
f_6(x)= \alpha \lVert A x \rVert_2 + \beta e_1^T A x,
\end{equation}
where $X \subset \mathbb{R}^n$ is the box $[-2, 2]^n$, $\alpha \geq \beta$ are
two real numbers, $A$ is a given symmetric positive definite matrix, and $e_1 =
(1, 0, \ldots 0)$. For $\alpha \geq \beta$, by  Cauchy–Schwarz inequality, we
have $f_6(x)\geq (\alpha-\beta)\|Ax\|_2$ and therefore the optimal value is 0 with corresponding optimal solution $x_* = 0$. We consider a
special case from this problem class taking $\alpha=4$ and $\beta=3$. This
instance was used as a test problem in \cite{skajaaLimitedMemoryBFGS,
SagastizabalCompositeProximalBundle2013, deoliveiraDoublyStabilizedBundle2016}
and called TiltedNorm. For all methods, the initial point is $x_0 = (1, 1,
\ldots, 1)$ and the maximum number of iterations is $500000$. The results show that Adaptive GPB OneCut and Adaptive GPB TwoCuts variants exhibit good numerical performance. In contrast to GPB 1C, which fails to reach the prescribed tolerance within the maximum number of iterations, both Adaptive GPB variants converge in all tested cases. For $n=200$, GPB 2C attains the fastest running times; however, the execution times of both Adaptive GPB variants are close and remain highly competitive when compared to the other methods.
For $n=50$, Adaptive GBP OneCut variant presents the best result.

\section{Conclusion}

We proposed an adaptive extension of the Generic Proximal bundle method from \cite{liang2023unified} called Adaptive GPB where step sizes
vary (are remained constant for the next cycle, increase, or decrease) along the iterations of the method depending on the outputs of a subroutine.
We have shown the complexity of the method.
Adaptive GPB 
relies on a blackbox BU which can be implemented
in many different ways. We propose
to use the Adaptive GPB OneCut and TwoCuts variants
which have the advantage of yielding subproblems \eqref{def:xj} which are solved extremely fast when $h=0$ or when
$h$ is the indicator of a box-constrained set.
We have shown the superiority of these two variants compared
to two other bundle methods from the literature on six nondifferentiable optimization problems from the literature. For static (with fixed step size $\lambda$) GPB, an extension 
based on a OneCut blackbox for stochastic problems
was proposed and studied in \cite{bundlegpbstochastic}. It would be
interesting to also extend Adaptive GPB for stochastic problems. 

\section*{Data availability}
There is no data related to this publication.


\bibliographystyle{plain}
\bibliography{last_versions/biblio}

\section*{Appendix A. Technical results}\label{appAsecB}

 We start recalling two results from \cite{guigues2024universal} and \cite{liang2023unified}. Lemma \ref{lemmaa6_glm24} is 
a corollary of
Lemma A.6 in \cite{guigues2024universal} 
with $\chi=0$
and Lemma \ref{lemma4.7_lm24} below is 
Lemma 4.7 from \cite{liang2023unified}.

\begin{lemma} \label{lemmaa6_glm24}
Let $j$ be an iteration of a cycle with step size $\lambda$ and first iteration $i$. Then
\begin{equation}
t_j-\frac{\bar{\varepsilon}}{4} \leq \tau^{j-i}\left(t_i-\frac{\bar{\varepsilon}}{4}\right)
\end{equation}
where $\tau$ is given by
\begin{equation}
\tau \geq \frac{u_\lambda}{1+u_\lambda}, \text{ with } u_\lambda:=\frac{4 \lambda\left[2 \overline M_f^2+ (\bar{\varepsilon}/2) \overline L_f\right]}{ \bar{\varepsilon}/2}.
\end{equation}\end{lemma}

In what follows, we denote by $j_k+1$ the first iteration of cycle $k$.

\begin{lemma}\label{lemma4.7_lm24}
Consider a given cycle $k$ with stepsize $\lambda$.
Then we have $t_{j_k+1} \leq \bar{t}$ where 
{\small
\begin{equation}
\bar{t}:=\overline{M}_f^2+4\left(\overline{L}_f+2\right)\left(\max \left\{1,2 \lambda \overline{L}_f\right\} d_0+\lambda \overline{M}_f\right)^2 .
\end{equation}}
\end{lemma}

\par {\textbf{Proof of Lemma \ref{lemk2b}.}} (a) 
Denoting $\varepsilon=\bar \varepsilon/2$, it follows from Lemma \ref{lemmaa6_glm24} (see the Appendix) that
as long as
\begin{equation}\label{deftauk}
\tau \geq  1- \left(1+ \frac{4\lam(2 {\overline M}_f^2 + \varepsilon {\overline L}_f)  }{\varepsilon}\right)^{-1},
\end{equation}
we have
\begin{equation}\label{ineq:tj-recur}
         t_{j+1}-\frac{\varepsilon}2 \le \tau \left(t_j -\frac{\varepsilon}2\right)
     \end{equation}
     for all inner iteration $j$ of the cycle.
Moreover, using Lemma \ref{lemma4.7_lm24} (see the Appendix), if $i$ is the first iteration of the cycle then
\begin{equation}\label{upptbar}
\begin{array}{lcl}
t_i & \leq &  
{\overline M}_f^2 +
4(2+{\overline L}_f)(d_0\max(1,2 \lambda {\overline L}_f)+\lambda\,{\overline M}_f)^2 \\
& \leq &  
\bar t := {\overline M}_f^2 +
4(2+{\overline L}_f)(d_0\max(1,2{\overline \lambda}\,{\overline L}_f)+{\overline \lambda}\,{\overline M}_f)^2,
\end{array}
\end{equation}
where we have used $\lambda \leq \overline \lambda$.     
When $\lambda \leq \lambda_*$, relation \eqref{deftauk}
is satisfied and once
we enter the subroutine with such $\lambda$ (satisfying \eqref{deftauk}) we end the corresponding cycle 
in at most
$k_1$ iterations, i.e., 
for $k \geq k_1$ we have 
\begin{equation}\label{successkbar}
t_{k+i-1} \leq {\bar \varepsilon}/2.
\end{equation}
Indeed, relation \eqref{successkbar} comes
from the relations
$$
\begin{array}{lcl}
t_{k+i-1} &\leq &{\bar \varepsilon}/4 + \tau^{{k}-1}(t_{i}-{\bar \varepsilon}/4) \mbox{ by }\eqref{ineq:tj-recur},\\
&\leq &{\bar \varepsilon}/4 + \tau^{{k}-1}({\bar t}-{\bar \varepsilon}/4) \mbox{ using }\eqref{upptbar},\\
&\leq &{\bar \varepsilon}/4 + \tau^{{k_1}-1}({\bar t}-{\bar \varepsilon}/4) \mbox{ for }k \geq k_1,\\
&\leq &{\bar \varepsilon}/4 + \tau^{k_0}({\bar t}-{\bar \varepsilon}/4),\\
&\leq &{\bar \varepsilon}/2
\mbox{ using }
\eqref{defk1},
\end{array}
$$
for $k \geq k_1$.
\par (b) The relation
$\lambda \geq {\underline \lambda}$ comes from (a) and the fact that
$\lambda$ is halved whenever it is decreased. 
To get a maximal number of bad cycles
 we need to start from a cycle starting with $\overline \lambda$.
Halving   $k$ times
$\lambda$
we get $\overline \lambda/2^{k}$
which needs to be $\geq {\underline \lambda}$, from which we deduce the upper bound $k_2$ on the number
of consecutive bad cycles.











\end{document}

%% file: def.tex
\usepackage[utf8]{inputenc}
\usepackage[T1]{fontenc}

\usepackage{amsmath}
\usepackage{amsfonts}
\usepackage{latexsym}
\usepackage{amssymb}
\usepackage{hyperref}
\usepackage{xcolor}
\hypersetup{
    colorlinks,
    linkcolor={red!50!black},
    citecolor={blue!50!black},
    urlcolor={blue!80!black}
}
\usepackage{booktabs} 

\newtheorem{thm}{Theorem}[section]
\newtheorem{theorem}[thm]{Theorem}

\newtheorem{lemma}[thm]{Lemma}

\newtheorem{proposition}[thm]{Proposition}

\newtheorem{remark}[thm]{Remark}

\newcommand{\beq}{\begin{equation}}
\newcommand{\eeq}{\end{equation}}
\newcommand{\beqa}{\begin{eqnarray}}
\newcommand{\eeqa}{\end{eqnarray}}
\newcommand{\beqas}{\begin{eqnarray*}}
\newcommand{\eeqas}{\end{eqnarray*}}
\newcommand{\bi}{\begin{itemize}}
\newcommand{\ei}{\end{itemize}}

\newcommand{\R}{\mathbb{R}}

\newcommand{\lam}{{\lambda}}

\newcommand{\inner}[2]{\langle #1,#2\rangle}

\newcommand{\argmin}{\mathrm{argmin}\,}

\newcommand{\dom}{\mathrm{dom}\,}

\newcommand{\bConv}[1]{\overline{\mbox{\rm Conv}}\,(\R^{#1})}

\newcommand{\tx}{\tilde x}
\newcommand{\ty}{\tilde y}

\newcommand{\mConv}[1]{\overline{\mbox{\rm Conv}}_\mu\,(\R^{#1})}